\documentclass[11pt]{article}
\usepackage[cp1251]{inputenc}
\usepackage{amsfonts,amsmath,amssymb,theorem}
\usepackage[pdftex, colorlinks, urlcolor=blue, pdfborder={0 0 0 [0 0]}]{hyperref}

\newtheorem{theorem}{Theorem}
\newtheorem{lemma}{Lemma}
\newtheorem{cor}{Corollary}
\newtheorem{prp}{Proposition}
\newtheorem{remark}{Remark}
\newtheorem{dfn}{Definition}

\newenvironment{proof}
{\vspace{1pt}\par{\sl%
Proof.\,\ }}%
{\noindent\vspace{1pt}}

\begin{document}

\title{On a generalization of the It\^{o}-Wentzell formula for system of generalized It\^{o}'s SDE  and  the stochastic first integral}

\author{Elena V. Karachanskaya \\
{Pacific National University, Russia, Khabarovsk}}


\maketitle

\begin{abstract}
Generalization of the It\^{o}-Wentzell formula for  the generalized It\^{o}'s SDE (It\^{o}'s GSDE) system with a non-centered measure is constructed   on the basis of the stochastic kernel of integral transformation. The  It\^{o}'s GSDE system for the kernel the solution of which is the kernel of the integral invariant, is formed. This invariant is connected with the solution of the It\^{o}'s GSDE system non-centered measure.  The concept of a stochastic first integral of the It\^{o}'s GSDE system with non-centered measure is introduced and conditions  that when being performed the  random function is the first integral of the set It\^{o}'s GSDE system are defined.
\end{abstract}

\section*{Introduction}

The specific approach to the stochastic integration based on the fact that not the point but its neighborhoods is reflected  according to some quite  different  from the traditional rule, based on Malliavin's calculus, was presented in {\rm{\cite{D_78,D_89,D_02}}}.

The transfer to the stochastic equations for definite realizations in based on comparing the dynamics of the space structures element marked by an integer to the dynamics of the trajectory's realization for which this integer is the initial value. This possibility is stipulated  by the  proximity of the trajectories, in the sense of probability, beginning the domain of the initial neighborhood for each of the temporary intervals.

The investigation of open systems that are characterized by the presence of random disturbances in them both continuous described in Wiener's processes and jump-like ones, presented in Poisson's processes is urgent nowadays. At the same time in case of random disturbances some characteristics of the system must be presented. And the problems of constructing the program control in the open systems arise. When a disturbance is caused by Wiener processes only there are some types of algorythms for constructing program control. The Poisson component of the random disturbance still not let to do program control for such systems. The present generalization of the It\^{o}-Wentzell formula solves this problem \cite{11_KchUpr}.

\section{The stochastic volume }

Let  ${\bf x}(t)$ -- is the random dynamic process and this process is solution of the next SDE system:
\begin{equation}\label{Ayd01.1}
\begin{array}{c}
  dx_{i}(t)= \displaystyle a_{i}\Bigl(t;{\bf x}(t) \Bigr)\, dt
  +\sum\limits_{k=1}^{m}
b_{i,k}\Bigl(t;{\bf x}(t)\Bigr)\, dw_{k}(t) +
\int\limits_{\mathbb{R}(\gamma)}g_{i}\Bigl(t;{\bf x}(t);\gamma\Bigr)\nu(dt;d\gamma), \\
  {\bf x}(t)={\bf x}\Bigl(t;{\bf x}(0) \Bigr) \Bigr|_{t=0}={\bf
  x}(0), \ \ \ \ \ \ i=\overline{1,n}, \ \ \ \ \ \ t\geq 0,
\end{array}
\end{equation}
where  ${\bf w} (t)$   is $m$--dimensional Wiener process,  $\nu(t;\Delta\gamma)$ --  is homogeneous on $t$ non centered Poisson measure.

We suppose that the coefficients  ${ a}(t;{\bf x})$, ${ b}(t;{\bf x})$,   and ${g}(t;{\bf x};\gamma)$  are bounded and continuous together with own derivations  $
\displaystyle \frac{\partial a_{i}(t;{\bf x})}{\partial x_{j}}$,\,
$\displaystyle \frac{\partial b_{i,k}(t;{\bf x})}{\partial
x_{j}}$,\, $\displaystyle \frac{\partial g_{i}(t;{\bf
x};\gamma)}{\partial x_{j}}$,\, $i,j=\overline{1,n}
$   respect to  $(t;{\bf x};\gamma)$.

These conditions are sufficient for existing random variables
\begin{equation*}
\textit{J}_{i,j}(t)= \mathop {\operatorname{l}
.i.m.}\limits_{\left|\Delta \right|= \Delta_{j}\rightarrow
0}\displaystyle \frac{x_{i}\Bigl(t; {\bf x}(0)+\Delta
\Bigr)-x_{i}\Bigl(t; {\bf x}(0)\Bigr)}{\Delta_{j}},
\end{equation*}
that can be defined by solving the system (\ref{Ayd01.1}) and the system of stochastic equation {\rm{\cite{GS_68}}}:
\begin{equation}\label{Ayd2.2}
\begin{array}{c}
  d\textit{J}_{i,j}(t)=\textit{J}_{l,j}(t)\, \biggl[\displaystyle
\frac{\partial a_{i}(t;{\bf x}(t))}{\partial x_{l}(0)}\, dt +
\beta \,\displaystyle \frac{\partial b_{i,k}(t;{\bf
x}(t))}{\partial
x_{l}(0)}\, dw_{k}(t) + \\
+\displaystyle \int\limits_{\mathbb{R}(\gamma)}
\frac{\partial g_{i}(t;{\bf x}(t))}{\partial x_{l}(0)}\nu(dt;d\gamma) \biggr], \\
\textit{J}_{i,j}(t)=\textit{J}_{i,j}\Bigl(t; {\bf x}(0) \Bigr)
\Bigr|_{t=0}=\delta_{i,j}=\left\{
\begin{array}{ll}
  0, & i\neq j,  \\
  1, & i=j.
\end{array}
  \right.
\end{array}
\end{equation}

Let  $\textit{J}(t)= \textit{J}\Bigl(t; {\bf x}(0)\Bigr)=\det\Bigl(\textit{J}_{ij}\Bigl(t; {\bf x}(0)\Bigr)\Bigr)$, $i,j=\overline{1,n}$. Then  the random variable $\textit{J}(t)$  can be considerer a Jacobian of transition from  ${\bf x}(0)$  to ${\bf x}(t;{\bf x}(0))$.

By means of random variables $\textit{J}_{i,j}(t)$   it is possible to construct different random manifolds ${\cal M}\Bigl(t; {\bf x}(t;\lambda ) \Bigr)$ ($\lambda\in {\cal
D} \subset \mathbb{R}^{r}, \  {\bf x}(t
;\lambda )\in\mathbb{R}^{n}, \  r \leq n $), having  compared them to configuration manifolds of the initial manifold ${\cal M}\Bigl(0; {\bf x}(0;\lambda ) \Bigr)$ ($\lambda\in {\cal
D} \subset \mathbb{R}^{r}, \  {\bf x}(0;\lambda )={\bf
x}(\lambda )\in\mathbb{R}^{n}, \ r \leq n $).

\begin{dfn}\label{Aydo2}{\rm{\cite{D_89}}}
Let's call the structure   $
d\Gamma(t)=\textit{J}(t) \,d\Gamma\Bigl({\bf x}(0)\Bigr)
$, where  $
d\Gamma({\bf x})= \displaystyle \prod \limits_{i=1}^{n} dx_{i}$,
an element of stochastic volume that it generated by some random process ${\bf x}\Bigl(t; {\bf x}(0;
\lambda)\Bigr)$, where $\textit{J}(t)$ is a Jacobian built of elements  $\textit{J}_{ij}\Bigl(t; {\bf x}(0)\Bigr)$.
\end{dfn}

 According to definition \ref{Aydo2}  a random structure is introduced:
\begin{equation}\label{Ayd2.3}
\begin{array}{c}
  {\cal I}_{\Gamma(t)}(t)=\underbrace{\int \cdots
\int}\limits_{\Gamma (t)} f(t;{\bf z}) \, d \Gamma ({\bf z})
  =\underbrace{\int \cdots \int} \limits_{\Gamma (0)}f\Bigl(t;{\bf
x}(t;{\bf y})\Bigr) \, \textit{J}(t)\, d \Gamma ({\bf y}),
\end{array}
\end{equation}
where ${\bf x}(t;{\bf y}) $ is a stochastic process with random initial condition  ${\bf y}$.

in so doing the integration over a stochastic volume  $\Gamma(t)$  is carried out Riemann's integration sums, mean square limit concept for any continuous functions by ${\bf z}$, and in general, random functions $f(t;{\bf z})$.

 Generically, the structure \eqref{Ayd2.3} is a random variable. The statement of the  classic integral calculus can be associated with random structure \eqref{Ayd2.3}.  In this sense  the equality  \eqref{Ayd2.3} is  an analogue for substituting variables in the classic calculus. Here, a  Jacobian of transformation $\textit{J}_{i,j}(t)$   is a random value comparising the  solutions of GSDE system (\ref{Ayd2.2}).

Let us determine the conditions for the invariance (or conservation) of the stochastic volume.

\section{Stochastic kernel of the stochastic integral invariant}

Let's consider the It\^{o}'s GSDE  system of the following kind (\ref{Ayd01.1}). Let's assume that  $\rho(t;{\bf x};\omega)$ --   is a random function which is measured with respect to  $\sigma$-algebras flow  $\Bigl\{\mathcal{F}\Bigr\}_{0}^{T}$,
$\mathcal{F}_{t_{1}}\subset\mathcal{F}_{t_{2}}$, $t_{1}<t_{2}$. This flow is  adapted with the processes   ${\bf w} (t)$ and  $\nu(t;\Delta\gamma)$. Each function $f(t;{\bf x})$  from the set of local bounded functions which have second bounded   derivatives with respect to  ${\bf x}$  the next relations hold:
\begin{equation}\label{Aydyad1}
\displaystyle\int\limits_{\mathbb{R}^{n}}\rho(t;{\bf x})f(t;{\bf
x})d\Gamma({\bf x})=\int\limits_{\mathbb{R}^{n}}\rho(0;{\bf
y})f(t;{\bf x}(t;{\bf y}))d\Gamma({\bf y})
\end{equation}
\begin{equation}\label{Aydyad2}
\displaystyle\int\limits_{\mathbb{R}^{n}}\rho(0;{\bf
x})d\Gamma({\bf x})=1,
\end{equation}
\begin{equation}\label{Aydyad3}
\begin{array}{c}
\displaystyle  \lim\limits_{|{\bf x}|\to \infty}\rho(0;{\bf x})=0,
\ \ \ d\Gamma({\bf x})=\prod\limits_{i=1}^{n}dx_{i}.
\end{array}
\end{equation}
Here  ${\bf x}(t;{\bf y})$  is the solution of (\ref{Ayd01.1}).

If we have $f(t;{\bf x})=1$,  then taking into consideration the conditions (\ref{Aydyad1}) and (\ref{Aydyad2}) we'll have the equality
\begin{equation}\label{usl-inv}
\displaystyle\int\limits_{\mathbb{R}^{n}}\rho(t;{\bf
x})d\Gamma({\bf x})=\int\limits_{\mathbb{R}^{n}}\rho(0;{\bf
y})d\Gamma({\bf y})=1.
\end{equation}
It means that there exists a random functional, which retains the constant value  for the random function  $\rho(t;{\bf x})=\rho(t;{\bf
x};\omega)$:
\begin{equation}\label{ro1}
\displaystyle\int\limits_{\mathbb{R}^{n}}\rho(t;{\bf
x})d\Gamma({\bf x})=1.
\end{equation}
 We consider this functional a stochastic volume.
Then the equation \eqref{Aydyad1} under conditions \eqref{Aydyad2} and \eqref{Aydyad3} is a stochastic integral invariant for function $f(t;{\bf x})$.

\begin{dfn}
Let's call the non-negative function $\rho(t;{\bf x})$   a stochastic kernel or stochastic density of the stochastic integral invariant if equation \eqref{Aydyad1}, \eqref{Aydyad2} and \eqref{Aydyad3}   are kept.
\end{dfn}

\begin{remark}
But the essential distinction that let to consider the invariance of the random volume on the basis of integral invariant kernel in {\rm\cite{D_89,D_02}} to that  there is a functional factor in \eqref{Aydyad1}. Thus, the notion of integral invariant kernel for the system of the ordinary differential equations can be considered a particular case of the introduced one in {\rm\cite{D_89,D_02}} if we take $f(t;{\bf x})=1$ and examine the integration by determined volume, having excluded the randomness preset in Wiener and Poisson processes.
\end{remark}

Let arbitrary function $f(t;{\bf x})$ has the first derivation with respect to $t$ and the second derivations wish respect to  ${\bf x}$. We determine relations for function  $\rho(t;{\bf x})$ which implies that  $\rho(t;{\bf x})$ is integral invariant kernel for $f(t;{\bf x})$.

Let  ${\bf x}(t)$ -- is solution of (\ref{Ayd01.1}) and  $f(t;{\bf x}(t))$ is random function. We apply the generalized It\^{o}'s formula {\rm{\cite{GS_68}}}:
\begin{equation}\label{Ayd2}
\begin{array}{c}
 \displaystyle d_{t}f(t;{\bf x}(t))=\Bigl[ \frac{\partial f(t;{\bf x}(t)) }{\partial t}+
 \sum\limits_{i=1}^{n}a_{i}(t;{\bf x}(t))\frac{\partial f(t;{\bf x}(t)) }{\partial x_{i}}+\Bigr.\\
\Bigl.+ \displaystyle  \frac{1}{2} \sum\limits_{i=1}^{n}
 \sum\limits_{j=1}^{n}
 \sum\limits_{k=1}^{m}b_{i\,k}(t;{\bf x}(t))b_{j\,k}(t;{\bf x}(t))
 \frac{\partial^{\,2} f(t;{\bf x}(t)) }{\partial x_{i}x_{j}}\Bigr]dt +\\
 +\displaystyle \sum\limits_{i=1}^{n}
 \sum\limits_{k=1}^{m}b_{i\,k}(t;{\bf x}(t))\frac{\partial f(t;{\bf x}(t)) }{\partial x_{i}} dw_{k}(t)+\\
 +\displaystyle \int\limits_{\mathbb{R}(\gamma)}
 \Bigl[ f\left(t;{\bf x}(t)+g(t;{\bf x}(t);\gamma)\right)- f(t;{\bf x}(t))
 \Bigr]\nu(dt;d\gamma).
 \end{array}
\end{equation}

\begin{remark}
Later on for notational  simplicity  we'll not use the  summation  symbol if we have double occurred indexes.
\end{remark}

Let's differentiate both parts of equation \eqref{Aydyad1} wish respect to $t$ taking into account that $f(t;{\bf x})$ of the left part is a determined function while $f(t;{\bf x}(t;{\bf y}))$ of the right part is a random process.
\begin{equation*}
\begin{array}{c}
  \displaystyle\int\limits_{\mathbb{R}^{n}}\Bigl(f(t;{\bf
x})d_{t}\rho(t;{\bf x})+\rho(t;{\bf x})\frac{\partial f(t;{\bf
x})}{\partial t}dt\Bigr)d\Gamma({\bf
x})= \int\limits_{\mathbb{R}^{n}}\rho(0;{\bf y})d_{t}f(t;{\bf x}(t;{\bf
y}))d\Gamma({\bf y}).
\end{array}
\end{equation*}
Then, taking into account both (\ref{Aydyad1}) and  (\ref{Ayd2}) we get:
\begin{equation}\label{Aydyad5}
\begin{array}{c}
  \displaystyle\int\limits_{\mathbb{R}^{n}}\Bigl(f(t;{\bf
x})d_{t}\rho(t;{\bf x})+\rho(t;{\bf x})\frac{\partial f(t;{\bf
x})}{\partial t}dt\Bigr)d\Gamma({\bf
x})=\\
=\displaystyle\int\limits_{\mathbb{R}^{n}}\rho(0;{\bf y})d_{t}
f(t;{\bf x}(t;{\bf y}))d\Gamma({\bf y})
=\displaystyle\int\limits_{\mathbb{R}^{n}}\rho(t;{\bf
x})d_{t}f(t;{\bf
x})d\Gamma({\bf x})= \\
=\displaystyle\int\limits_{\mathbb{R}^{n}}d\Gamma({\bf
x})\rho(t;{\bf x})\cdot \biggl[ \Bigl[\frac{\partial f(t;{\bf
x}(t)) }{\partial t}+
 a_{i}(t;{\bf x})\frac{\partial f(t;{\bf x}) }{\partial x_{i}}+\Bigr.\biggr.\\
\Bigl.+ \displaystyle  \frac{1}{2} b_{i\,k}(t;{\bf
x})b_{j\,k}(t;{\bf x})
 \frac{\partial^{\,2} f(t;{\bf x}) }{\partial x_{i} \partial x_{j}}\Bigr]dt
 +\displaystyle
 b_{i\,k}(t;{\bf x})\frac{\partial f(t;{\bf x}) }{\partial x_{i}} dw_{k}(t)+\\
\biggl. +\displaystyle \int\limits_{\mathbb{R}(\gamma)}
 \Bigl[ f\left(t;{\bf x}+g(t;{\bf x};\gamma)\right)- f(t;{\bf x})
 \Bigr]\nu(dt;d\gamma)\biggr].
\end{array}
\end{equation}
Let's consider the integral
\begin{equation}\label{Aydy1}
I=\displaystyle\int\limits_{\mathbb{R}^{n}}d\Gamma({\bf
x})\rho(t;{\bf x})
 f\left(t;{\bf x}+g(t;{\bf x};\gamma)\right).
\end{equation}
We change variables:
\begin{equation}\label{Aydx1y}
{\bf x}+g(t;{\bf x};\gamma)={\bf y},
\end{equation}
Having taken into account the transition from  ${\bf x}$  to  ${\bf y}$, we get
$
{\bf x}={\bf y}-g(t;{\bf x};\gamma)={\bf y}-g(t;{\bf
x}^{-1}(t;{\bf y};\gamma);\gamma)
$.  And the integral $I$  becomes:
\begin{equation}\label{Aydy2}
\begin{array}{c}
  I=\displaystyle \int\limits_{\mathbb{R}^{n}}d\Gamma({\bf
y})\rho\left(t;{\bf y}-g(t;{\bf x}^{-1}(t;{\bf
y};\gamma);\gamma)\right)
  \cdot f(t;{\bf y})\cdot D\left( {\bf x}^{-1}(t;{\bf y};\gamma)  \right),
\end{array}
\end{equation}
where  $D\left( {\bf x}^{-1}(t;{\bf y};\gamma)\right) $  is a Jacobian of transition from ${\bf x}$  to  ${\bf y}$.

Taking into accout  (\ref{Aydy2}) and integration over the $\mathbb{R}^{n}$  we have:
\begin{equation}\label{Aydy3}
\begin{array}{c}
  \displaystyle\int\limits_{\mathbb{R}^{n}}d\Gamma({\bf
x})\rho(t;{\bf x})
 \displaystyle \int\limits_{\mathbb{R}(\gamma)}
 \Bigl[ f\left(t;{\bf x}+g(t;{\bf x};\gamma)\right)- f(t;{\bf x})
 \Bigr]\nu(dt;d\gamma)= \\
  =\displaystyle\int\limits_{\mathbb{R}^{n}}d\Gamma({\bf
x})\rho(t;{\bf x})
 \displaystyle \int\limits_{\mathbb{R}(\gamma)}
f\left(t;{\bf x}+g(t;{\bf
x};\gamma)\right)\nu(dt;d\gamma)-\\
-\displaystyle\int\limits_{\mathbb{R}^{n}}d\Gamma({\bf
x})\rho(t;{\bf x})
 \displaystyle \int\limits_{\mathbb{R}(\gamma)} f(t;{\bf
x})
\nu(dt;d\gamma)= \\
=\displaystyle \int\limits_{\mathbb{R}^{n}}d\Gamma({\bf x})
  \int\limits_{\mathbb{R}(\gamma)}  \Bigl(\Bigr.\rho\left(t;{\bf x}-g(t;{\bf x}^{-1}(t;{\bf
x};\gamma);\gamma)\right) \cdot\\
\cdot
 f(t;{\bf x})\cdot D\left( {\bf x}^{-1}(t;{\bf x};\gamma)  \right)\Bigl.\Bigr) \nu(dt;d\gamma)-\\
 -\displaystyle\int\limits_{\mathbb{R}^{n}}d\Gamma({\bf
x})\rho(t;{\bf x})
 \displaystyle \int\limits_{\mathbb{R}(\gamma)} f(t;{\bf x}) \nu(dt;d\gamma)=\\
=\displaystyle \int\limits_{\mathbb{R}^{n}}d\Gamma({\bf
x})f(t;{\bf x})\int\limits_{\mathbb{R}(\gamma)}
\Bigl[\rho\left(t;{\bf x}-g(t;{\bf x}^{-1}(t;{\bf
x};\gamma);\gamma)\right)
 \cdot \Bigr.\\
 \cdot\Bigl. D\left( {\bf x}^{-1}(t;{\bf x};\gamma)  \right)-\rho(t;{\bf
 x})\Bigr]\nu(dt;d\gamma).
\end{array}
\end{equation}
With due regard for (\ref{Aydyad2}) we calculate the following integrals using integration  by parts:
\begin{equation}\label{Aydi01}
\displaystyle\int\limits_{\mathbb{R}^{n}}d\Gamma({\bf
x})\rho(t;{\bf x})a_{i}(t;{\bf x})\frac{\partial f(t;{\bf x})
}{\partial x_{i}},
\end{equation}
\begin{equation}\label{Aydi02}
\displaystyle\int\limits_{\mathbb{R}^{n}}d\Gamma({\bf
x})\rho(t;{\bf x})b_{i\,k}(t;{\bf x})\frac{\partial f(t;{\bf x})
}{\partial x_{i}},
\end{equation}
\begin{equation}\label{Aydi03}
\displaystyle\int\limits_{\mathbb{R}^{n}}d\Gamma({\bf
x})\rho(t;{\bf x})b_{i\,k}(t;{\bf x})b_{j\,k}(t;{\bf x})
 \frac{\partial^{\,2} f(t;{\bf x}) }{\partial x_{i} \partial x_{j}}.
\end{equation}
For equation (\ref{Aydi01}) we get:
\begin{equation*}
\begin{array}{c}
  \displaystyle\int\limits_{\mathbb{R}^{n}}d\Gamma({\bf
x})\rho(t;{\bf x})a_{i}(t;{\bf x})\frac{\partial f(t;{\bf x})
}{\partial x_{i}}
=\displaystyle\int\limits_{-\infty}^{+\infty}\prod\limits_{j=1}^{n}dx_{j}\rho(t;{\bf
x})a_{i}(t;{\bf x})\frac{\partial f(t;{\bf x}) }{\partial x_{i}}= \\
  =\displaystyle\int\limits_{-\infty}^{+\infty}\prod\limits_{j=1}^{n-1}dx_{j}
  \int\limits_{-\infty}^{+\infty}\rho(t;{\bf
x})a_{i}(t;{\bf x})\frac{\partial f(t;{\bf x}) }{\partial
x_{i}}dx_{i}.
\end{array}
\end{equation*}
Then we calculate the inner integral using integration in parts, having taken into account  (\ref{Aydyad3}):
\begin{equation*}
\begin{array}{c}
\displaystyle\int\limits_{-\infty}^{+\infty}\rho(t;{\bf
x})a_{i}(t;{\bf x})\frac{\partial f(t;{\bf x}) }{\partial
x_{i}}dx_{i}=\\
=\rho(t;{\bf x})a_{i}(t;{\bf x})f(t;{\bf
x})\biggl|_{-\infty}^{+\infty}-\displaystyle\int\limits_{-\infty}^{+\infty}f(t;{\bf
x})\displaystyle\frac{\partial \left( \rho(t;{\bf x})a_{i}(t;{\bf
x})\right) }{\partial
x_{i}}dx_{i} =\\
=-\displaystyle\int\limits_{-\infty}^{+\infty}f(t;{\bf
x})\displaystyle\frac{\partial \left( \rho(t;{\bf x})a_{i}(t;{\bf
x})\right) }{\partial x_{i}}dx_{i} .
\end{array}
\end{equation*}
Therefore, we get
\begin{equation}\label{Aydi10}
\begin{array}{c}
\displaystyle\int\limits_{\mathbb{R}^{n}}d\Gamma({\bf
x})\rho(t;{\bf x})a_{i}(t;{\bf x})\frac{\partial f(t;{\bf x})
}{\partial x_{i}}
=\displaystyle\int\limits_{-\infty}^{+\infty}\prod\limits_{j=1}^{n}dx_{j}\rho(t;{\bf
x})a_{i}(t;{\bf x})\frac{\partial f(t;{\bf x}) }{\partial x_{i}}= \\
  =\displaystyle\int\limits_{-\infty}^{+\infty}\prod\limits_{j=1}^{n-1}dx_{j}
  \int\limits_{-\infty}^{+\infty}\rho(t;{\bf
x})a_{i}(t;{\bf x})\frac{\partial f(t;{\bf x}) }{\partial
x_{i}}dx_{i}=\\
=-\displaystyle\int\limits_{-\infty}^{+\infty}\prod\limits_{j=1}^{n}dx_{j}
\displaystyle\int\limits_{-\infty}^{+\infty}f(t;{\bf
x})\displaystyle\frac{\partial \left( \rho(t;{\bf x})a_{i}(t;{\bf
x})\right) }{\partial
x_{i}}dx_{i} =\\
=
-\displaystyle\int\limits_{\mathbb{R}^{n}}d\Gamma({\bf
x})f(t;{\bf x})\displaystyle\frac{\partial \left( \rho(t;{\bf
x})a_{i}(t;{\bf x})\right) }{\partial x_{i}}
\end{array}
\end{equation}

Similarly we calculate the second integral (\ref{Aydi02}) and the third integral (\ref{Aydi03}) by using  twice  the integration  in parts:
\begin{equation}\label{Aydi20}
\begin{array}{c}
  \displaystyle\int\limits_{\mathbb{R}^{n}}d\Gamma({\bf
x})\rho(t;{\bf x})b_{i\,k}(t;{\bf x})\frac{\partial f(t;{\bf x})
}{\partial x_{i}}
  =-\displaystyle\int\limits_{\mathbb{R}^{n}}d\Gamma({\bf
x})f(t;{\bf x})\frac{\partial\rho(t;{\bf x})b_{i\,k}(t;{\bf x})
}{\partial x_{i}},
\end{array}
\end{equation}
\begin{equation}\label{Aydi30}
\begin{array}{c}
  \displaystyle\int\limits_{\mathbb{R}^{n}}d\Gamma({\bf
x})\rho(t;{\bf x})b_{i\,k}(t;{\bf x})b_{j\,k}(t;{\bf x})
 \frac{\partial^{\,2} f(t;{\bf x}) }{\partial x_{i} \partial x_{j}}=\\
 =\displaystyle\int\limits_{\mathbb{R}^{n}}d\Gamma({\bf x})f(t;{\bf
x})
 \frac{\partial^{\,2} \left(\rho(t;{\bf x})b_{i\,k}(t;{\bf x})b_{j\,k}(t;{\bf x})\right) }{\partial x_{i} \partial
 x_{j}}.
\end{array}
\end{equation}

In equation (\ref{Aydyad5}) we move all terms to the right part and   taking into account  (\ref{Aydy3}), (\ref{Aydi10}), (\ref{Aydi20}) and  (\ref{Aydi30}) we obtain the following  equation:
\begin{equation}\label{Aydi}
\begin{array}{c}
 0= \displaystyle\int\limits_{\mathbb{R}^{n}}d\Gamma({\bf x})f(t;{\bf
x})\cdot \biggl[-d_{t}\left(\rho(t;{\bf x})\right)-\frac{\partial
f(t;{\bf x}(t)) }{\partial t}+\biggr.\\
+ \displaystyle\frac{\partial f(t;{\bf x}(t)) }{\partial
t}-\frac{\partial\rho(t;{\bf x})b_{i\,k}(t;{\bf x}) }{\partial
x_{i}}dw_{k}(t)+\\
+\Bigl(\Bigr.-\displaystyle\frac{\partial \left( \rho(t;{\bf
x})a_{i}(t;{\bf x}(t))\right) }{\partial x_{i}}
+\displaystyle\frac{1}{2}\frac{\partial^{\,2} \left(\rho(t;{\bf
x})b_{i\,k}(t;{\bf x})b_{j\,k}(t;{\bf x})\right) }{\partial x_{i} \partial x_{j}}\Bigl.\Bigr)dt + \\
 \biggl.+\displaystyle\int\limits_{\mathbb{R}(\gamma)} \Bigl[\rho\left(t;{\bf
x}-g(t;{\bf x}^{-1}(t;{\bf x};\gamma);\gamma)\right)
  \cdot D\left( {\bf x}^{-1}(t;{\bf x};\gamma)  \right)-\rho(t;{\bf
 x})\Bigr]\nu(dt;d\gamma)\biggr].
\end{array}
\end{equation}
To make the equation (\ref{Aydi}) true for any locally bounded function $f(t;{\bf x})$ obtaining the limited second-order derivatives the integral invariant $\rho(t;{\bf x})$ must be the solution for the following stochastic equation:
\begin{equation}\label{Aydii}
\begin{array}{c}
 d_{t}\rho(t;{\bf x})=
-\displaystyle\frac{\partial\rho(t;{\bf x})b_{i\,k}(t;{\bf x})
}{\partial x_{i}}dw_{k}(t) +\Bigl(-\displaystyle\frac{\partial
\left( \rho(t;{\bf x})a_{i}(t;{\bf x})\right) }{\partial
x_{i}}+\\
+\displaystyle\frac{1}{2}\frac{\partial^{\,2} \left(\rho(t;{\bf
x})b_{i\,k}(t;{\bf x})b_{j\,k}(t;{\bf x})\right) }{\partial x_{i} \partial x_{j}}\Bigr)dt + \\
 +\displaystyle\int\limits_{\mathbb{R}(\gamma)} \Bigl[\rho\left(t;{\bf
x}-g(t;{\bf x}^{-1}(t;{\bf x};\gamma);\gamma)\right)
 \cdot D\left( {\bf x}^{-1}(t;{\bf x};\gamma)  \right)-\rho(t;{\bf
 x})\Bigr]\nu(dt;d\gamma).
\end{array}
\end{equation}
The following terms must be kept:
\begin{equation}\label{Aydii1}
\begin{array}{c}
\rho(t;{\bf x})\Bigl|_{t=0}=\rho(0;{\bf x})\in C_{0}^{2},  \\
\lim\limits_{|{\bf x}|\to \infty}\rho(0;{\bf x})=0, \ \ \ \
\displaystyle\lim\limits_{|{\bf x}|\to
\infty}\frac{\partial\rho(0;{\bf x})}{\partial x_{i}}=0.
 \end{array}
\end{equation}
Thus, we get the invariance conditions of a stochastic volume and the following theorem is proved.
\begin{theorem}\label{thro1}
Let   ${\bf x}(t)$,
${\bf x}\in\mathbb{R}^{n}$ --   is the solution of It\^{o}'s GSDE  system:
\begin{equation*}
\begin{array}{c}
  dx_{i}(t)= \displaystyle a_{i}\Bigl(t;{\bf x}(t) \Bigr)\, dt
  +\sum\limits_{k=1}^{m}
b_{i,k}\Bigl(t;{\bf x}(t)\Bigr)\, dw_{k}(t) +
\int\limits_{\mathbb{R}(\gamma)}g_{i}\Bigl(t;{\bf x}(t);\gamma\Bigr)\nu(dt;d\gamma), \\
  {\bf x}(t)={\bf x}\Bigl(t;{\bf x}(0) \Bigr) \Bigr|_{t=0}={\bf
  x}(0), \ \ \ \ \ \ i=\overline{1,n}, \ \ \ \ \ \ t\geq 0,
\end{array}
\end{equation*}
where ${\bf w} (t)$ -- is $m$-dimentional  Wiener process,  $\nu(t;\Delta\gamma)$ -- is homogeneous on $t$ non centered Poisson measure.  Assume that  $\rho(t;{\bf x})$   is a random function which is measurable function with respect to   $\sigma$-algebras flow $\Bigl\{\mathcal{F}\Bigr\}_{0}^{T}$,
$\mathcal{F}_{t_{1}}\subset\mathcal{F}_{t_{2}}$, $t_{1}<t_{2}$, which adapting     ${\bf w} (t)$ and $\nu(t;\Delta\gamma)$  processes and with respect to any function $f(t;{\bf x})\in\mathfrak{S}$  from the set of locally bounded functions having the second bounded  derivatives in ${\bf x}$.  The function  $\rho(t;{\bf x})$ is the stochastic kernel of the stochastic integral invariant \eqref{Aydyad1}
 for arbitrary locally bounded function  $f(t;{\bf x})\in\mathfrak{S}$  if it's a solution of It\^{o}'s GSDE  system \eqref{Aydii}  meeting the initial conditions \eqref{Aydii1}.
\end{theorem}

\section{The generalization of It\^{o}-Wentzell formula}

The equation (\ref{Aydyad1}) under conditions (\ref{Aydyad2}), \eqref{Aydyad3} and equation (\ref{Aydii})  for the integral invariant kernel conformed to a determinate function $f(t;{\bf x})$. Let's assume, that analogical equation for random function ${\bf
z}(t;{\bf x})$  holds:
\begin{equation}\label{Aydz3p}
\begin{array}{c}
  \displaystyle\int\limits_{\mathbb{R}^{n}}\rho(t;{\bf x}){\bf z}(t;{\bf
x})d\Gamma({\bf x})=
  \displaystyle\int\limits_{\mathbb{R}^{n}}\rho(0;{\bf y}){\bf
z}(t;{\bf x}(t;{\bf y}))d\Gamma({\bf y}),
\end{array}
\end{equation}
here  ${\bf x}(t;{\bf y})$ is a solution of (\ref{Ayd01.1}).

Let's consider a compound stochastic process ${\bf z}\left(t;{\bf x}(t;{\bf
y})\right)\in \mathbb{R}^{n_{o}}$, where ${\bf x}(t;{\bf y})$ --  is a solution of (\ref{Ayd01.1}) where the process ${\bf z}(t;{\bf
x})$ --  is  a solution of  the system:
\begin{equation}\label{Aydz1}
\begin{array}{c}
  d_{t}{\bf z}(t;{\bf x})=\Pi(t;{\bf x})dt+D_{k}(t;{\bf x})dw_{k}(t)
+
  \displaystyle\int\limits_{R(\gamma)}\nu(dt;d\gamma)G(t;{\bf x};\gamma).
  \end{array}
\end{equation}

As far as the random functions  $\Pi(t;{\bf x})$, $D_{k}(t;{\bf
x})$, $G(t;{\bf x};\gamma)$   defined  on $\mathbb{R}^{n_{o}}$ are concerned we can assume that they are continuous  and  bounded together with their own  derivatives with respect to all variables. These  functions are  measurable with respect to  $\sigma$-algebras flow  $\mathcal{F}_{t}$, which adapted with processes ${\bf w}(t)$  and $\nu(t;\Delta\gamma)$  from equation (\ref{Ayd01.1}).

Relying on the equation for the  stochastic integral invariant. We can construct   the differentiation law for the compound stochastic process  ${\bf z}\left(t;{\bf x}(t;{\bf
y})\right)$.

Let us consider  the integral   $
\displaystyle\int\limits_{\mathbb{R}^{n}}\rho(t;{\bf x}){\bf
z}(t;{\bf x})d\Gamma({\bf x})
$.  As we do integration over the  $\mathbb{R}^{n}$ and ${\bf x}\in \mathbb{R}^{n}$, we use equation  (\ref{Aydz3p}) and having changed the places of the equation, we get :
\begin{equation}\label{Aydz3}
\begin{array}{c}
  \displaystyle\int\limits_{\mathbb{R}^{n}}\rho(0;{\bf y}){\bf
z}(t;{\bf x}(t;{\bf y}))d\Gamma({\bf y})=
  \int\limits_{\mathbb{R}^{n}}\rho(t;{\bf x}){\bf z}(t;{\bf
x})d\Gamma({\bf x}).
\end{array}
\end{equation}
Lrt's differentiate both parts of (\ref{Aydz3}) with respect to $t$:
\begin{equation}\label{Aydz4}
\begin{array}{c}
  \displaystyle\int\limits_{\mathbb{R}^{n}}\rho(0;{\bf y})d_{t}{\bf
z}(t;{\bf x}(t;{\bf y}))d\Gamma({\bf
y})= \\
 \displaystyle\int\limits_{\mathbb{R}^{n}}\Bigl(\Bigr.\rho(t;{\bf x})d_{t}{\bf
z}(t;{\bf x})+{\bf z}(t;{\bf x})d_{t}\rho(t;{\bf x}) -
D_{k}(t;{\bf x}) \displaystyle\frac{\partial\rho(t;{\bf
x})b_{i\,k}(t;{\bf x}) }{\partial x_{i}}dt\Bigl.\Bigr)d\Gamma({\bf
x}) ,
\end{array}
\end{equation}

Taking into account that integration is carried out over the $\mathbb{R}^{n}$ and the variable is changed we should keep in ind this substitution (\ref{Aydx1y}) when integration over the curve  $R(\gamma)$.

As the function $\rho(t;{\bf x})$ --  is the integral invariant kernel \eqref{Aydz3p}  we apply the  theorem~\ref{thro1}.  Substitute (\ref{Aydii}) and  (\ref{Aydz1}) to the right  part of (\ref{Aydz4}):
$$
I_{1}= \displaystyle \int\limits_{\mathbb{R}^{n}}\Bigl(\rho(t;{\bf
x})d_{t}{\bf z}(t;{\bf x})+{\bf z}(t;{\bf x})d_{t}\rho(t;{\bf x})
-D_{k}(t;{\bf x}) \displaystyle\frac{\partial\rho(t;{\bf
x})b_{i\,k}(t;{\bf x})
}{\partial x_{i}}dt\Bigr)d\Gamma({\bf x})=
$$
$$  = \displaystyle\int\limits_{\mathbb{R}^{n}}d\Gamma({\bf x})\biggl\{\rho(t;{\bf x})
  \Bigl[\Bigr.\Pi(t;{\bf x})dt+D_{k}(t;{\bf x})dw_{k}(t)
+
$$
$$
+\displaystyle\int\limits_{R(\gamma)}G(t;{\bf x}+g(t;{\bf
x}^{-1}(t;{\bf
x};\gamma);\gamma)\nu(dt;d\gamma)\Bigl.\Bigr]+
$$
$$
+{\bf z}(t;{\bf x})\biggl[-\displaystyle\frac{\partial\rho(t;{\bf
x})b_{i\,k}(t;{\bf x}) }{\partial
x_{i}}dw_{k}(t)+
$$
$$
+\Bigr.\displaystyle\Bigl(\Bigr.-\displaystyle\frac{\partial
\left( \rho(t;{\bf x})a_{i}(t;{\bf x})\right) }{\partial
x_{i}}+
\displaystyle\frac{1}{2}\frac{\partial^{\,2} \left(\rho(t;{\bf
x})b_{i\,k}(t;{\bf x})b_{j\,k}(t;{\bf x})\right) }{\partial x_{i} \partial x_{j}}\Bigl.\Bigr)dt +
$$
$$
 +\displaystyle\int\limits_{\mathbb{R}(\gamma)} \Bigl[\Bigr.\rho\left(t;{\bf
x}-g(t;{\bf x}^{-1}(t;{\bf x};\gamma);\gamma)\right)
 \cdot  D\left( {\bf x}^{-1}(t;{\bf x};\gamma)  \right)-\rho(t;{\bf
 x})\Bigl.\Bigr]\nu(dt;d\gamma)\Bigl.\Bigl.\biggr]-
 $$
 $$
 -D_{k}(t;{\bf
x}) \displaystyle\frac{\partial\rho(t;{\bf x})b_{i\,k}(t;{\bf x})
}{\partial x_{i}}dt\biggr\}.
$$
Collect similar terms:
\begin{equation}\label{Aydz5}
\begin{array}{c}
I_{1}=\displaystyle\int\limits_{\mathbb{R}^{n}}d\Gamma({\bf
x})\biggl(\rho(t;{\bf x})\Pi(t;{\bf x}) -{\bf z}(t;{\bf
x})\displaystyle\frac{\partial \left( \rho(t;{\bf x})a_{i}(t;{\bf
x})\right) }{\partial x_{i}}+\biggr.\\
-D_{k}(t;{\bf x}) \displaystyle\frac{\partial\rho(t;{\bf
x})b_{i\,k}(t;{\bf x}) }{\partial x_{i}}
+\displaystyle\frac{1}{2}{\bf z}(t;{\bf x})\frac{\partial^{\,2}
\left(\rho(t;{\bf x})b_{i\,k}(t;{\bf x})b_{j\,k}(t;{\bf x})\right)
}{\partial x_{i}
\partial x_{j}}\biggl.\biggr)dt+\\
+\displaystyle\int\limits_{\mathbb{R}^{n}}d\Gamma({\bf
x})\biggl(\rho(t;{\bf x})D_{k}(t;{\bf x}) -{\bf z}(t;{\bf
x})\displaystyle\frac{\partial \left( \rho(t;{\bf
x})b_{i\,k}(t;{\bf x})\right) }{\partial
x_{i}}\biggr)dw_{k}(t)+\\
+\displaystyle\int\limits_{\mathbb{R}^{n}}d\Gamma({\bf
x})\biggl[\rho(t;{\bf x}) \cdot\int\limits_{R(\gamma)}G(t;{\bf
x}+g(t;{\bf x}^{-1}(t;{\bf
x};\gamma);\gamma))\nu(dt;d\gamma) +\Bigr.\\
+{\bf z}(t;{\bf
x})\displaystyle\int\limits_{R(\gamma)}\Bigl[\rho\left(t;{\bf
x}-g(t;{\bf x}^{-1}(t;{\bf x};\gamma);\gamma)\right)
 \cdot D\left( {\bf x}^{-1}(t;{\bf x};\gamma)  \right)-\rho(t;{\bf
 x})\Bigr]\nu(dt;d\gamma)\Bigl.\biggr].
\end{array}
\end{equation}
 And because of (\ref{Aydi10}), (\ref{Aydi20}), (\ref{Aydi30})  we have:
\begin{equation*}
\begin{array}{c}
\displaystyle\int\limits_{\mathbb{R}^{n}}d\Gamma({\bf x}){\bf
z}(t;{\bf x})\frac{\partial \left(\rho(t;{\bf x})a_{i}(t;{\bf
x})\right) }{\partial x_{i}} =
-\displaystyle\int\limits_{\mathbb{R}^{n}}d\Gamma({\bf
x})\rho(t;{\bf x})a_{i}(t;{\bf x})\displaystyle\frac{\partial
 {\bf z}(t;{\bf x})
 }{\partial x_{i}},
\end{array}
\end{equation*}
\begin{equation*}
\begin{array}{c}
  \displaystyle\int\limits_{\mathbb{R}^{n}}d\Gamma({\bf
x})D_{k}(t;{\bf x})\frac{\partial\left(\rho(t;{\bf
x})b_{i\,k}(t;{\bf x})\right) }{\partial x_{i}}=
  -\displaystyle\int\limits_{\mathbb{R}^{n}}d\Gamma({\bf
x})\rho(t;{\bf x})b_{i\,k}(t;{\bf x})\frac{\partial D_{k}(t;{\bf
x}) }{\partial x_{i}},
\end{array}
\end{equation*}
\begin{equation*}
\begin{array}{c}
\displaystyle\int\limits_{\mathbb{R}^{n}}d\Gamma({\bf x}){\bf
z}(t;{\bf x})
 \frac{\partial^{\,2} \left(\rho(t;{\bf x})b_{i\,k}(t;{\bf x})b_{j\,k}(t;{\bf x})\right) }{\partial x_{i} \partial
 x_{j}}=
  \\
 = \displaystyle\int\limits_{\mathbb{R}^{n}}d\Gamma({\bf
x})\rho(t;{\bf x})b_{i\,k}(t;{\bf x})b_{j\,k}(t;{\bf x})
 \frac{\partial^{\,2} {\bf z}(t;{\bf x}) }{\partial x_{i} \partial
 x_{j}}.
\end{array}
\end{equation*}
Calculate the last integral in (\ref{Aydz5}):
\begin{equation}\label{Ayde6}
\begin{array}{c}
  I_{2}=\displaystyle\int\limits_{\mathbb{R}^{n}}d\Gamma({\bf
x}){\bf z}(t;{\bf x})
  \cdot\rho\Bigl(t;{\bf x}-g(t;{\bf x}^{-1}(t;{\bf
x};\gamma);\gamma)\Bigr)D({\bf x}^{-1}(t;{\bf x};\gamma)).
\end{array}
\end{equation}
Let's introduce the variables' substitution:
$$
\begin{array}{c}
  {\bf x}-g(t;{\bf x}^{-1}(t;{\bf x};\gamma);\gamma)={\bf y}, \\
  {\bf x}={\bf y}+g(t;{\bf x}^{-1}(t;{\bf x};\gamma);\gamma)={\bf y}+g(t;{\bf
  y};\gamma).
\end{array}
$$
Let  $D_{o}({\bf x}^{-1}(t;{\bf y};\gamma))$  be a Jacobian of transformation from  the vector ${\bf x}$  to the vector ${\bf
y}$. Then, by  virtue of (\ref{Aydy1}) and further formal substitution of the integration variable's notation, we get:
\begin{equation}\label{Ayde7}
\begin{array}{c}
  I_{2}=\displaystyle\int\limits_{\mathbb{R}^{n}}d\Gamma({\bf y}){\bf
z}\Bigl(t;{\bf y}+g(t;{\bf
  y};\gamma)\Bigr)
  \cdot\rho(t;{\bf y})D_{o}({\bf x}^{-1}(t;{\bf y};\gamma))D({\bf
x}^{-1}(t;{\bf x};\gamma))= \\
 =\displaystyle\int\limits_{\mathbb{R}^{n}}d\Gamma({\bf y}){\bf
z}\Bigl(t;{\bf y}+g(t;{\bf
  y};\gamma)\Bigr)\rho(t;{\bf y})
  =\displaystyle\int\limits_{\mathbb{R}^{n}}d\Gamma({\bf x}){\bf
z}\Bigl(t;{\bf x}+g(t;{\bf
  x};\gamma)\Bigr)\rho(t;{\bf x}).
\end{array}
\end{equation}
 As a result, the right part of (\ref{Aydz4})  becomes:
\begin{equation}\label{Ayde8}
\begin{array}{c}
I_{1}=\displaystyle\int\limits_{\mathbb{R}^{n}}d\Gamma({\bf
x})\rho(t;{\bf x})\biggl\{\Bigl( D_{k}(t;{\bf x}) +b_{i\,k}(t;{\bf
x})\frac{\partial {\bf z}(t;{\bf x}) }{\partial x_{i}}
\Bigr)dw_{k}(t)+\biggr.\\+
 \Bigl(\Pi(t;{\bf x})+a_{i}(t;{\bf
x})\displaystyle\frac{\partial
 {\bf z}(t;{\bf x})
 }{\partial x_{i}}+ b_{i\,k}(t;{\bf
x})\frac{\partial
 D_{k}(t;{\bf x})
 }{\partial x_{i}}+\\
 +\displaystyle\frac{1}{2}b_{i\,k}(t;{\bf x})b_{j\,k}(t;{\bf x})
 \frac{\partial^{\,2} {\bf z}(t;{\bf x}) }{\partial x_{i} \partial
 x_{j}}\Bigr)dt+\\
 +\displaystyle\int\limits_{R(\gamma)}G(t;{\bf x}+g(t;{\bf
x}^{-1}(t;{\bf x};\gamma);\gamma)\nu(dt;d\gamma)+\\
+
 \displaystyle\int\limits_{R(\gamma)}\Bigl[{\bf z}\Bigl(t;{\bf x}+g(t;{\bf
  x};\gamma)\Bigr)
  -{\bf z}(t;{\bf
 x})\Bigr]\nu(dt;d\gamma)\biggl.\biggr\}.
 \end{array}
\end{equation}
In equation (\ref{Aydz4}) we transfer all the terms to the right part. And having taken  (\ref{Aydyad1}) into account, we have:
\begin{equation*}
\begin{array}{c}
0=\displaystyle\int\limits_{\mathbb{R}^{n}}d\Gamma({\bf
y})\rho(0;{\bf y})\biggl\{-d_{t}{\bf z}(t;{\bf x}(t;{\bf
y}))+\\
+\biggr.\Bigl( D_{k}(t;{\bf x}(t;{\bf y}))+b_{i\,k}(t;{\bf x}(t;{\bf
y}))\displaystyle\frac{\partial {\bf z}(t;{\bf
x}) }{\partial x_{i}} \Bigr)dw_{k}(t)+\\
+ \Bigl(\Pi(t;{\bf x}(t;{\bf y}))+a_{i}(t;{\bf
x})\displaystyle\frac{\partial
 {\bf z}(t;{\bf x})
 }{\partial x_{i}}
 +b_{i\,k}(t;{\bf
x})\frac{\partial
 D_{k}(t;{\bf x})
 }{\partial x_{i}}+\Bigr.\\
 +\Bigl.\displaystyle\frac{1}{2}b_{i\,k}(t;{\bf x}(t;{\bf y}))b_{j\,k}(t;{\bf x}(t;{\bf y}))
 \frac{\partial^{\,2} {\bf z}(t;{\bf x}(t;{\bf y})) }{\partial x_{i} \partial
 x_{j}}\Bigr)dt+\\
 +\displaystyle\int\limits_{R(\gamma)}G(t;{\bf x}(t;{\bf y})+g(t;{\bf
  x}(t;{\bf y});\gamma)\nu(dt;d\gamma)+\\
 +
 \displaystyle\int\limits_{R(\gamma)}\Bigl[{\bf z}\Bigl(t;{\bf x}(t;{\bf y})+g(t;{\bf
  x}(t;{\bf y});\gamma)\Bigr)-
{\bf z}(t;{\bf
 x}(t;{\bf y}))\Bigr]\nu(dt;d\gamma)\biggl.\biggr\}.
 \end{array}
\end{equation*}
Therefore,  we'll have the differentiation law for the compound stochastic process as  the theorem.
\begin{theorem}\label{thro2}
Let  ${\bf x}(t;{\bf y})\in\mathbb{R}^{n} $  is the solution of SDE system  \eqref{Ayd01.1},  ${\bf z}(t;{\bf
x})$ is  the solution of SDE system \eqref{Aydz1} and  ${\bf z}\left(t;{\bf x}(t;{\bf
y})\right)\in \mathbb{R}^{n_{o}}$  is the  random process. The random functions that are coefficients of equations   \eqref{Ayd01.1} and \eqref{Aydz1}, which are defined on  $\mathbb{R}^{n_{o}}$ and  $\mathbb{R}^{n}$ respectively. This functions  are continuous  and  bounded together with their  own  derivatives respect to all variables. Also these  functions are measurable respect to  $\sigma$-algebras flow  $\mathcal{F}_{t}$, which is adapted with processes  ${\bf w}(t)$  and  $\nu(t;\Delta\gamma)$  from equation \eqref{Ayd01.1}. Then the compound stochastic process ${\bf z}\left(t;{\bf x}(t;{\bf
y})\right)$  is the solution of SDE system:
\begin{equation}\label{Ayde9}
\begin{array}{c}
d_{t}{\bf z}(t;{\bf x}(t;{\bf y}))=\displaystyle\Bigl(
D_{k}(t;{\bf x}(t;{\bf y})) +b_{i\,k}(t;{\bf x}(t;{\bf
y}))\frac{\partial {\bf z}(t;{\bf
x}(t;{\bf y})) }{\partial x_{i}} \Bigr)dw_{k}(t)+\\
 +\Bigl(\Pi(t;{\bf x}(t;{\bf y}))+a_{i}(t;{\bf
x}(t;{\bf y}))\displaystyle\frac{\partial
 {\bf z}(t;{\bf x}(t;{\bf y}))
 }{\partial x_{i}}
 +b_{i\,k}(t;{\bf
x}(t;{\bf y}))\displaystyle\frac{\partial
 D_{k}(t;{\bf x}(t;{\bf y}))
 }{\partial x_{i}}+\Bigr.\\
 +\Bigl.\displaystyle\frac{1}{2}b_{i\,k}(t;{\bf x}(t;{\bf y}))b_{j\,k}(t;{\bf x}(t;{\bf y}))
 \frac{\partial^{\,2} {\bf z}(t;{\bf x}(t;{\bf y})) }{\partial x_{i} \partial
 x_{j}}\Bigr)dt+\\
 +\displaystyle\int\limits_{R(\gamma)}G(t;{\bf x}(t;{\bf y})+g(t;{\bf
  x}(t;{\bf y});\gamma)\nu(dt;d\gamma)+\\
 +
 \displaystyle\int\limits_{R(\gamma)}\Bigl[{\bf z}\Bigl(t;{\bf x}(t;{\bf y})+g(t;{\bf
  x}(t;{\bf y});\gamma)\Bigr)
  -{\bf z}(t;{\bf
 x}(t;{\bf y}))\Bigr]\nu(dt;d\gamma).
 \end{array}
\end{equation}
\end{theorem}

On the analogy of the well-known terminology let's call the formula (\ref{Ayde9}) as generalized It\^{o}-Wentzell  formula for the generalized SDE It\^{o} of non-centered  Poisson measure.
In addition we can settle one more statement.
\begin{prp}\label{thro3}
 If ${\bf z}(t;{\bf x}(t;{\bf y}))$ --  is the solution for equation \eqref{Ayde9} meeting the initial condition
\begin{equation*}
{\bf z}(t;{\bf x}(t;{\bf y}))\Bigr|_{t=0}={\bf z}(0;{\bf y}), \ \
\ {\bf z}(0;{\bf y})\in C_{o}^{2},
\end{equation*}
then the random function $\rho(t;{\bf x})$,  which is a function that measurable with   respect to $\sigma$-algebras flow  $\Bigl\{\mathcal{F}\Bigr\}_{0}^{T}$,
$\mathcal{F}_{t_{1}}\subset\mathcal{F}_{t_{2}}$, $t_{1}<t_{2}$, that is adapted with processes ${\bf w} (t)$  and   $\nu(t;\Delta\gamma)$, is the stochastic kernel of the stochastic integral invariant \eqref{Aydz3p}.
\end{prp}

\section{The stochastic first integral}\label{PintSt}
In the article \cite{D_78} the concept of the first integral for It\^{o}'s SDE  system (without Poisson part) has been introduced. In the article \cite{D_02} the concept of the stochastic first integral for It\^{o}'s GSDE  system with  centered measure has been given. Let's introduce the similar concept for the non-centered  Poisson measure.
\begin{dfn}\label{Ayddf1}
Let  random function  $u(t;{\bf x};\omega)$   and   solution ${\bf x}(t)$ of the It\^{o}'s GSDE  system \eqref{Ayd01.1} are defined  on the same    probability space. The  function  $u(t;{\bf x};\omega)$  is called the stochastic first integral for It\^{o}'s GSDE  system \eqref{Ayd01.1} with  non centered Poissin measure, if condition
$$
u\Bigl(t;{\bf x}(t; {\bf x}(0));\omega\Bigr)=u\Bigl(0;{\bf
x}(0);\omega\Bigr)
$$
is held with probability is equaled to 1 for each solution  ${\bf x}(t;{\bf x}(0);\omega)$  of the system \eqref{Ayd01.1}.
\end{dfn}

We set conditions for the function $u(t;{\bf x};\omega)$  which imply that   $u(t;{\bf x};\omega)$  is the stochastic first integral for It\^{o}'s GSDE  system \eqref{Ayd01.1}.
\begin{lemma}\label{l1}
If the function $\rho(t;{\bf x})$  is the stochastic kernel of the  integral invariant of order $n$ of the stochastic process ${\bf x}(t)$ outgoing from a random point ${\bf x}(0)={\bf y}$, then for any $t$ it complies with the equation $\rho\Bigl(t;{\bf x}(t;{\bf y})\Bigr)\textit{J}(t;{\bf y})=\rho(0;{\bf y})$, where $\textit{J}(t;{\bf y})=\textit{J}(t;{\bf x}(0))$  is  a transition Jacobian from    ${\bf x}(t)$ to ${\bf x}(0)={\bf y}$.
\end{lemma}
\begin{proof}
Let's proceed to the equation \eqref{usl-inv}. If we substitute the  variables of the integral we'll get the integration is fulfilled according one and the same random volume:
\begin{equation*}
1=\int\limits_{\mathbb{R}^{n}}\rho(0;{\bf
y})d\Gamma({\bf y})=\displaystyle\int\limits_{\mathbb{R}^{n}}\rho(t;{\bf
x})d\Gamma({\bf x})=\displaystyle\int\limits_{\mathbb{R}^{n}}\rho\Bigl(t;{\bf x}(t;{\bf y})\Bigr)\textit{J}(t;{\bf
y})d\Gamma({\bf y}).
\end{equation*}
\end{proof}
As the random process ${\bf x}(t)$ is defined within the extended phase space of general dimension $\mathbb{R}^{n+1}$, then to determine the uniqueness of the trajectory the system of differential equations for stochastic  kernel  should consists of at least $n+1$  equations:
\begin{equation}\label{sys-yad}
\left\{
\begin{array}{l}
 d_{t}\rho_{l}(t;{\bf x})=
-\displaystyle\frac{\partial\rho_{l}(t;{\bf x})b_{i\,k}(t;{\bf x})
}{\partial x_{i}}dw_{k}(t) +\Bigl(-\displaystyle\frac{\partial
\left( \rho_{l}(t;{\bf x})a_{i}(t;{\bf x})\right) }{\partial
x_{i}}+\\
+\displaystyle\frac{1}{2}\frac{\partial^{\,2} \left(\rho_{l}(t;{\bf
x})b_{i\,k}(t;{\bf x})b_{j\,k}(t;{\bf x})\right) }{\partial x_{i} \partial x_{j}}\Bigr)dt + \\
 +\displaystyle\int\limits_{\mathbb{R}(\gamma)} \Bigl[\rho_{l}\left(t;{\bf
x}-g(t;{\bf x}^{-1}(t;{\bf x};\gamma);\gamma)\right)
 \cdot D\left( {\bf x}^{-1}(t;{\bf x};\gamma)  \right)-\rho_{l}(t;{\bf
 x})\Bigr]\nu(dt;d\gamma), \\
\rho_{l}(t;{\bf x}(t))\Bigl.\Bigr|_{t=0}=\rho_{l}(0;{\bf x}(0))=\rho_{l}(0;{\bf y}), \ \ \ \ \  l=\overline{1, n+1}.
\end{array}
\right.
\end{equation}

As you know the collection of kernels is called complete if any other function that is the integral invariant kernel of the $n$-order  can be   presented as the function of this collection' elements.
\begin{theorem}\label{th-yad}
The GSDE system  \eqref{Ayd01.1} possesses the complete collection of kernels, which consists of  $(n+1)$  functions and each of these functions is the solution for equation \eqref{Aydii}.
\end{theorem}

\begin{proof}
Let  $\rho_{l}(t;{\bf x})$,  $l=\overline{1,m}$, $m\geq n+1$ -- are the integral invariant kernels of \eqref{Aydyad1}.
As follows from lemma \ref{l1} that for any $l\neq n+1$
the  ratio $\dfrac{\rho_{l}(t;{\bf x}(t;{\bf y}))}
{\rho_{n+1}(t;{\bf x}(t;{\bf y}))}$
is a constant depending only on  ${\bf x}(0)={\bf y}$  for  any solution  ${\bf x}(t)$  of system \eqref{Ayd01.1}:
$
\dfrac{\rho_{l}(t;{\bf x}(t;{\bf y}))}
{\rho_{n+1}(t;{\bf x}(t;{\bf y}))}=\dfrac{\rho_{l}(0;{\bf y})}
{\rho_{n+1}(0;{\bf y})}.
$
Let's construct the functions $\theta_{l}(t;{\bf x})= \rho_{l}(t;{\bf x})\,\rho_{n+1}^{-1}(t;{\bf
x})$, $l=\overline{1,n}$,  taking into account the linear independence of the functions $\rho_{l}(0;{\bf y})$  and  $\rho_{n+1}(0;{\bf y})$.

Depending on the specific  approach, described in the Introduction and the obtained independence of the functions $\theta_{l}(t;{\bf x})$ with $t=0$ it was possible to construct  a one-to-one correspondence:
\begin{equation}\label{inv-1}
x_{i}=q_{i}(\theta_{l},\theta_{2},\ldots,\theta_{n}).
\end{equation}

Let's introduce the following notation just to shorten the entries:
$$
\overrightarrow{\theta}=
\left\{\theta_{l},\theta_{2},\ldots,\theta_{n} \right\},$$
$$\overrightarrow{q}=\left\{q_{l},q_{2},\ldots,q_{n} \right\}
\in {\mathbb R}^{n},
$$
$$
\overrightarrow{\rho}(t;{\bf x})=\left\{\rho_{1}(t;{\bf
x}),\rho_{2}(t;{\bf x}),\ldots, \rho_{3}(t;{\bf x})   \right\}
\in {\mathbb R}^{n}.
$$
Because of the condition \eqref{inv-1} for some $s\geq1$  and any
other function
$
\chi(t;{\bf x})=\rho_{n+s}(t;{\bf x})\,\rho_{n+1}^{-1}(t;{\bf x})
$  at the time moment $t=0$ can be represented like this:
$$
\chi(0;{\bf y})=\rho_{n+s}(0;{\bf y})\,\rho_{n+1}^{-1}(0;{\bf y})=\overline{\psi} \Bigl(
\overrightarrow{q}(\overrightarrow{\theta})\Bigr).
$$
Then, consequently, for any function $\rho_{n+s}(t;{\bf x})$ being a kernel of the integral invariant  \eqref{Ayd01.1} and because the lemma \ref{l1} conclusion, we get:
$\theta_{l}(t;{\bf x})$:
$$
\chi(t;{\bf x}(t;{\bf  y}))=\rho_{n+s}(t;{\bf x}(t;{\bf
y}))\,\rho_{n+1}^{-1}(t;{\bf x}(t;{\bf  y}))=
$$
$$
=\rho_{n+s}(0;{\bf y})\,\rho_{n+1}^{-1}(0;{\bf y})=\chi(0;{\bf y})=
\overline{\psi} \biggl(
\overrightarrow{q}\Bigl(\rho_{n+1}^{-1}(t;{\bf x}(t;{\bf  y}))\cdot\overrightarrow{\rho}(t;{\bf x}(t;{\bf
y})) \Bigr)\biggr).
$$
It follows from this relation the for any  $ t\geq 0 $  and each $s\geq 1$:
$$
\rho_{n+s}(t;{\bf x}(t;{\bf  y}))=\rho_{n+1}(t;{\bf x}(t;{\bf
y}))\, \overline{\psi}  \biggl(
\overrightarrow{q}\Bigl(\rho_{n+1}^{-1}(t;{\bf x}(t;{\bf  y}))\cdot\overrightarrow{\rho}(t;{\bf x}(t;{\bf
y})) \Bigr)\biggr).
$$
Then the  complete collection of the kernels consists of $(n+1)$  linearly independent  functions.
\end{proof}

\begin{cor}\label{c1}
The complete collection of linear independent stochastic first integrals of system \eqref{Ayd01.1} consists of  $n$  functions.
\end{cor}
\begin{proof}
The function  $\widetilde{u}_{l}(t;{\bf x})=\rho_{l}(t;{\bf x})\rho_{n+1}^{-1}(t;{\bf x})$, possesses the properties described in the definition \ref{Ayddf1} and presents the stochastic first integral of the  system \eqref{Ayd01.1}. There is $n$ number
of such  functions.
\end{proof}

\begin{remark}
Let the random process  ${\bf x}(t)$ is the solution of the generalized  It\^{o}'s SDE  system that is presented like this:
\begin{equation*}
    {d}_{t}{\bf x}(t)=a(t;{\bf x}(t))dt + b(t;{\bf
    x}(t))d{\bf w}(t)+ dP(t,\Delta\gamma)=\widetilde{d}_{t}{\bf x}(t)+dP(t,\Delta\gamma).
\end{equation*}
Then  the generalized It\^{o}'s formula \eqref{Ayd2} can be taken down as an equation:
\begin{equation}\label{AydIt2}
       d_{t}f(t;{\bf x}(t))= \widetilde{d}_{t}f(t;{\bf
       x}(t))+\widetilde{d}_{t}P(t,\Delta\gamma),
\end{equation}
where  $\widetilde{d}_{t}f(t;{\bf x}(t))$ is the It\^{o}'s differential for the part $\widetilde{d}_{t}{\bf x}(t)$  and $\widetilde{d}_{t}P(t,\Delta\gamma)$   is the differential for the Poisson's component  $dP(t,\Delta\gamma)$.
\end{remark}

Let's construct the equation for $u(t;{\bf x})$  using the corollary \ref{c1} and the formula:
\begin{equation}\label{Ayd2.11}
\ln u_{s}(t;{\bf x})=\ln \rho_{s}(t;{\bf x})-\ln\rho_{l}(t;{\bf
x}).
\end{equation}
By using \eqref{Aydii},   the It\^{o}'s generalized formula and (\ref{AydIt2}) we'll do the differentiation of $\ln \rho(t;{\bf x})$:
\begin{equation}\label{Aydln}
\begin{array}{c}
  \displaystyle d_{t}\ln\rho(t;{\bf x})= \frac{1}{\rho(t;{\bf
x})}\widetilde{d}_{t}\rho(t;{\bf x})-\frac{1}{2\rho^{2}(t;{\bf
x})}\left(-\displaystyle\frac{\partial(\rho(t;{\bf
x})b_{i\,k}(t;{\bf
x})) }{\partial x_{i}}\right)^{2}dt +  \\
  +\displaystyle\int\limits_{R(\gamma)}\Bigl[\Bigr.\ln\left\{\right.
\rho_{s}\left( t;{\bf x}-g(t;{\bf
  x}(t;{\bf y});\gamma);\gamma\right)
  \cdot D\left( {\bf x}^{-1}(t;{\bf x};\gamma)  \right)
  \left.\right\}-\\
  -\ln\rho_{s}(t;{\bf x})\Bigl.\Bigr]\nu(dt;d\gamma).
\end{array}
\end{equation}
Now let us examine the sum of the first two summands, dropping the functions' arguments:
\begin{equation}\label{Aydlnp}
\begin{array}{c}
  S_{1}=\displaystyle \frac{1}{\rho}\, \widetilde{d}_{t}\rho-\frac{1}{2\rho^{2}}\,
  \left(-\displaystyle\frac{\partial(\rho b_{i\,k})}{\partial x_{i}}\right)^{2} =\\
  = \displaystyle\frac{1}{\rho} \biggl[\Bigl( -\frac{\partial (\rho a_{i})}{\partial x_{i}}+
  \frac{1}{2}\frac{\partial ^{\,2}(\rho b_{i\,k}b_{j\,k})}{\partial x_{i}\partial x_{j}}
  \Bigr)dt + \Bigl( -\displaystyle\frac{\partial(\rho b_{i\,k}) }{\partial
  x_{i}}\Bigr)dw_{k}(t)\biggr]
  -\\
  -\displaystyle\frac{1}{2\rho^{2}}\,\left(-\displaystyle\frac{\partial(\rho b_{i\,k})}{\partial
 x_{i}}\right)^{2}dt=\\
 = \displaystyle\biggl[\biggr.  - \frac{\partial  a_{i}}{\partial x_{i}}-a_{i}\frac{\partial \ln\rho}{\partial x_{i}} +
 \frac{1}{2\rho}\frac{\partial}{\partial x_{i} }\left(\rho\,\frac{\partial b_{i\,k}b_{j\,k}}{\partial
  x_{j}}+ b_{i\,k}b_{j\,k}\,\frac{\partial \rho}{\partial x_{i}} \right)-
  \\
  -\displaystyle\frac{1}{2\rho^{2}}\,\left(\displaystyle\rho\,\frac{\partial b_{i\,k}}{\partial
 x_{i}}+b_{i\,k}\, \frac{\partial \rho}{\partial
 x_{i}}\right)^{2}\biggl.\biggr]dt
 -\left( b_{i\,k}\displaystyle\frac{\partial\ln\rho }{\partial
 x_{i}}+\frac{\partial b_{i\,k}}{\partial
 x_{i}}\right)dw_{k}(t)=\\
 =S_{2}dt-\left( b_{i\,k}\displaystyle\frac{\partial\ln\rho }{\partial
 x_{i}}+\frac{\partial b_{i\,k}}{\partial
 x_{i}}\right)dw_{k}(t).
\end{array}
\end{equation}
We transform the part $S_{2}$:
\begin{equation}\label{Aydlnp1}
\begin{array}{c}
S_{2}= - \displaystyle \frac{\partial  a_{i}}{\partial
x_{i}}-a_{i}\frac{\partial \ln\rho}{\partial x_{i}} +
\frac{\partial b_{i\,k}b_{j\,k}}{\partial
  x_{j}}\frac{\partial \ln\rho}{\partial x_{i}}  +\\
  +
  \displaystyle\frac{1}{2}\,\frac{\partial^{\,2}( b_{i\,k}b_{j\,k})}{\partial x_{i}\partial x_{j}}
    +\frac{1}{2}b_{i\,k}b_{j\,k}\,\frac{1}{\rho}\,\frac{\partial^{\,2} \rho}{\partial x_{i}\partial x_{j}}
  -\\
  -
  \displaystyle\frac{1}{2}\,\left(\displaystyle\frac{\partial b_{i\,k}}{\partial
 x_{i}}+b_{i\,k}\, \frac{\partial \ln\rho}{\partial
 x_{i}}\right)\left(\displaystyle\frac{\partial b_{j\,k}}{\partial
 x_{j}}+b_{j\,k}\, \frac{\partial \ln\rho}{\partial
 x_{j}}\right)=\\
 =- \displaystyle \frac{\partial  a_{i}}{\partial
x_{i}}-a_{i}\frac{\partial \ln\rho}{\partial x_{i}} +
\frac{\partial b_{i\,k}b_{j\,k}}{\partial
  x_{j}}\frac{\partial \ln\rho}{\partial x_{i}}  +\\
  +
  \displaystyle\frac{1}{2}\,\frac{\partial^{\,2}( b_{i\,k}b_{j\,k})}{\partial x_{i}\partial
  x_{j}}+\frac{1}{2}\,b_{i\,k}b_{j\,k}\frac{\partial^{\,2} \ln\rho}{\partial x_{i}\partial
  x_{j}}+
  +\displaystyle\frac{1}{2}\,b_{i\,k}b_{j\,k}\frac{\partial\ln\rho}{\partial x_{i}}\,\frac{\partial\ln\rho}{\partial
  x_{j}}-\\
  -\displaystyle\frac{1}{2}\left[\frac{\partial b_{i\,k}}{\partial x_{i}}\,
    \frac{\partial b_{j\,k}}{\partial x_{j}} +2\, b_{i\,k}\frac{\partial b_{j\,k}}{\partial
    x_{j}}\,\frac{\partial \ln\rho}{\partial x_{i}} +
    b_{i\,k}b_{j\,k}\,\frac{\partial\ln\rho}{\partial x_{i}}\,\frac{\partial\ln\rho}{\partial
  x_{j}}  \right]=\\
  =- \displaystyle \frac{\partial  a_{i}}{\partial
x_{i}}-a_{i}\frac{\partial \ln\rho}{\partial x_{i}} +
\frac{\partial b_{i\,k}b_{j\,k}}{\partial
  x_{j}}\,\frac{\partial \ln\rho}{\partial x_{i}}  +\\
  +\displaystyle
  \frac{1}{2}\,\frac{\partial^{\,2}( b_{i\,k}b_{j\,k})}{\partial x_{i}\partial
  x_{j}}+\frac{1}{2}\,b_{i\,k}b_{j\,k}\frac{\partial^{\,2} \ln\rho}{\partial x_{i}\partial
  x_{j}}
    -\displaystyle\frac{1}{2}  \frac{\partial b_{i\,k}}{\partial x_{i}}\,
    \frac{\partial b_{j\,k}}{\partial x_{j}} - b_{i\,k}\frac{\partial b_{j\,k}}{\partial
    x_{j}}\,\frac{\partial \ln\rho}{\partial x_{i}}.
   \end{array}
\end{equation}
Hence, by substituting (\ref{Aydlnp1}) into (\ref{Aydlnp}) and then into (\ref{Aydln}). We get  the following equation for  $\ln
\rho(t;{\bf x})$:
\begin{equation}\label{Ayd2.11a}
\begin{array}{c}
  d_{t}\ln\rho(t;{\bf x})=\biggl[\biggr.- \displaystyle
  \frac{\partial a_{i}(t;{\bf x})}{\partial x_{i}}-
  a_{i}(t;{\bf x})\frac{\partial \ln\rho(t;{\bf x})}{\partial x_{i}}-\\
  -\displaystyle \frac{1}{2}\,\frac{\partial b_{i\,k}(t;{\bf x})}
  {\partial x_{i}}\,\frac{\partial b_{j\,k}(t;{\bf x})}
  {\partial x_{j}}
  +
  \displaystyle \frac{1}{2}b_{i\,k}(t;{\bf x})b_{j\,k}(t;{\bf x})
  \frac{\partial^{2} \ln\rho(t;{\bf x})}{\partial x_{i}\partial x_{j}}+\\
  +\displaystyle\frac{\partial \left(b_{i\,k}(t;{\bf x})b_{j\,k}(t;{\bf x})\right)}
  {\partial x_{i}}\,
  \frac{\partial \ln\rho(t;{\bf x})}{\partial  x_{j}}+\\
  +\displaystyle\frac{1}{2}\frac{\partial^{2}\left( b_{i\,k}(t;{\bf x})b_{j\,k}(t;{\bf x})\right)}
  {\partial x_{i}\partial x_{j}}
  - b_{j\,k}(t;{\bf x})\displaystyle\frac{\partial b_{i\,k}(t;{\bf x})}
  {\partial x_{i}}\,\frac{\partial \ln\rho(t;{\bf x})}{\partial x_{j}}\biggl.\biggr] dt +\\
+\displaystyle\int\limits_{R(\gamma)}\Bigl[\Bigr.\ln\left\{\right.
\rho\left( t;{\bf x}-g(t;{\bf
  x}(t;{\bf y});\gamma);\gamma\right)
  \cdot D\left( {\bf x}^{-1}(t;{\bf x};\gamma)  \right)
  \left.\right\}-\\
  -\ln\rho(t;{\bf x})\Bigl.\Bigr]\nu(dt;d\gamma)
  -\biggl[\displaystyle \frac{\partial b_{i\,k}(t;{\bf x})}{\partial x_{i}}+
b_{j\,k}(t;{\bf x})\frac{\partial \ln\rho(t;{\bf x})}{\partial
x_{j}}
   \biggr]dw_{k}(t).
\end{array}
\end{equation}

Relaying upon the obtained formula  (\ref{Ayd2.11a}) we  can construct the equations for $\rho_{s}(t;{\bf x})$   and   $\rho_{l}(t;{\bf x})$, and having taken  (\ref{Ayd2.11}) into consideration we make a conclusion that the stochastic first integral  $u(t;{\bf x};\omega)$  of the generalized It\^{o}'s equation is the solution for this equation:
\begin{equation}\label{Ayd2.12}
\begin{array}{c}
  d_{t}u(t;{\bf x})=\biggl[\biggr. -  a_{i}(t;{\bf x})\displaystyle\frac{\partial u(t;{\bf x})}{\partial
x_{i}} +\displaystyle \frac{1}{2}\,b_{i\,k}(t;{\bf
x})b_{j\,k}(t;{\bf x})
  \frac{\partial^{2} u(t;{\bf x})}{\partial x_{i} \partial x_{j}} -\\
-  b_{i\,k}(t;{\bf x})\displaystyle\frac{\partial}{\partial
x_{i}}\left(b_{j\,k}(t;{\bf x})\frac{\partial u(t;{\bf
x})}{\partial x_{j}}\right)\biggl.\biggr]dt
  - b_{i\,k}(t;{\bf x})\displaystyle\frac{\partial u(t;{\bf x})}{\partial
x_{i}}\,dw_{k}(t)+ \\
+ \displaystyle\int\limits_{R(\gamma)}\Bigl[u\Bigl(t; {\bf
x}-g(t;{\bf
  x}^{-1}(t;{\bf x};\gamma));\gamma\Bigr)
 -
  u(t;{\bf x})\Bigr]\nu(dt;d\gamma).
\end{array}
\end{equation}

The concept of the first integral for  It\^{o}'s SDE  system as the non random function over arbitrary random perturbation realizations
has been introduce by the article {\rm{\cite{D_78}}}:
$$
u\Bigl(0;{\bf x}(0)\Bigr)=u\Bigl(t;{\bf x}(t; {\bf x}(0))\Bigr).
$$
We can also speak however about a stochastic first integral for  It\^{o}'s SDE  system. It should be observed that we can also speak about stochastic  (but not  only about determinate) first integral when Poisson's perturbations are missing.

In case of just Wiener's perturbations (It\^{o} classical equation)
the investigation of the properties of the first integral as a determined function was associated with stating the independence of such a function from realization ${\bf w}(t)$ {\rm{\cite{D_78}}}.
In the case under consideration from definition \ref{Ayddf1} the function  $u\Bigl(t;{\bf x}\Bigr)$ with its equation (\ref{Ayd2.12}), depends over Poisson's perturbation are added (It\^{o} generalized equation) such a requirement, as it follows from  equation (\ref{Ayd2.12}), leads to the following conditions.

\begin{theorem}\label{th-usl}
The random function $u\left(t;{\bf x}(t)\right) \in \mathcal{C}_{t,x}^{1,2}$ defined over the same  probability space as the  random process ${\bf x}(t)$ that is the solution for system \eqref{Ayd01.1}  is  the first integral of the system \eqref{Ayd01.1} iff the function $u\left(t;{\bf x}(t)\right)$ satisfies the terms  $\left.\mathcal{L}\right)$\label{uslL1}:
\begin{enumerate}
    \item $b_{i\,k}(t;{\bf x})\displaystyle\frac{\partial u(t;{\bf x})}{\partial
x_{i}}=0$, for all $k=\overline{1,m}$ (compensation of the Wiener's perturbations);
    \item $\displaystyle\frac{\partial u(t;{\bf x})}{\partial
t}+\displaystyle\frac{\partial u(t;{\bf x})}{\partial
x_{i}}\Bigl[a_{i}(t;{\bf x})-
\displaystyle\frac{1}{2}\,b_{j\,k}(t;{\bf x})\frac{\partial
b_{i\,k}(t;{\bf x})}{\partial x_{j}}\Bigr]=0$ (independente of the time);
     \item $u(t;{\bf x})-u\Bigl(t;{\bf x}+g(t;{\bf
    x};\gamma)\Bigr)=0$ for any $\gamma\in R(\gamma)$
         in the whole field of the process definition   (compensation of the Poisson's jumps).
\end{enumerate}
\end{theorem}

\begin{proof}
Let's examine the equation \eqref{Ayd2.12}, the solution of  which is the  function  $u\left(t;{\bf x}(t)\right)$. First let's carry everything to one side from the equals sing. Having taken into account that $d_{t}u(t;{\bf x})=\dfrac{\partial u(t;{\bf x})}{\partial t}\,dt$  at we get for any  $t$:
\begin{equation*}
\begin{array}{c}
  \biggl[\biggr.\dfrac{\partial u(t;{\bf x})}{\partial t}+   a_{i}(t;{\bf x})\displaystyle\frac{\partial u(t;{\bf x})}{\partial
x_{i}} -\displaystyle \frac{1}{2}\,b_{i\,k}(t;{\bf
x})b_{j\,k}(t;{\bf x})
  \frac{\partial^{2} u(t;{\bf x})}{\partial x_{i} \partial x_{j}} +\\
+   b_{i\,k}(t;{\bf x})\displaystyle\frac{\partial}{\partial
x_{i}}\left(b_{j\,k}(t;{\bf x})\frac{\partial u(t;{\bf
x})}{\partial x_{j}}\right)\biggl.\biggr]dt
  + b_{i\,k}(t;{\bf x})\displaystyle\frac{\partial u(t;{\bf x})}{\partial
x_{i}}\,dw_{k}(t)- \\
- \displaystyle\int\limits_{R(\gamma)}\Bigl[u\Bigl(t; {\bf
x}-g(t;{\bf
  x}^{-1}(t;{\bf x};\gamma));\gamma\Bigr)
 +
  u(t;{\bf x})\Bigr]\nu(dt;d\gamma)=0.
\end{array}
\end{equation*}
Hence, the multipliers must be equal to zero with $dt$,   $dw(t)$  and $\nu(dt;d\gamma)$. The Wiener's summand equaled zero:
For  the Wiener part we have:
\begin{equation}\label{usl-win}
b_{i\,k}(t;{\bf x})\displaystyle\frac{\partial u(t;{\bf x})}{\partial
x_{i}}=0 \ \ \ \ \textrm{for all}\ \ k=\overline{1,m}.
\end{equation}
For  Poisson's part we'll have:
\begin{equation}\label{usl-puass}
u\Bigl(t; {\bf x}-g(t;{\bf x}^{-1}(t;{\bf x};\gamma));\gamma\Bigr)-
u(t;{\bf x})=0.
\end{equation}
Now we re-arrange the equality  \eqref{usl-puass} and proceed to variables:
$${\bf y}={\bf x}-g(t;{\bf x}^{-1}(t;{\bf x};\gamma);\gamma).$$
But taking into account that ${\bf x}^{-1}(t;{\bf x};\gamma)$  is the notation an inverse function for  $f({\bf x})={\bf x}+g(t;{\bf x};\gamma)$,  we get convinced that  this term is equivalent to the following:
\begin{equation}\label{usl-puass-1}
u(t;{\bf x})-u\Bigl(t;{\bf x}+g(t;{\bf
    x};\gamma)\Bigr)=0 \ \ \ \textrm{for any} \ \ \ \gamma\in R(\gamma).
\end{equation}
Then, using the rule of differentiating the product and the term
\eqref{usl-win}, we obtain:
\begin{equation*}
\begin{array}{c}
  \dfrac{\partial u(t;{\bf x})}{\partial t}+
  a_{i}(t;{\bf x})\dfrac{\partial u(t;{\bf x})}{\partial
x_{i}} + \dfrac{1}{2}\,b_{i\,k}(t;{\bf
x})b_{j\,k}(t;{\bf x})
  \dfrac{\partial^{2} u(t;{\bf x})}{\partial x_{i} \partial x_{j}}\, -\\
-   b_{i\,k}(t;{\bf x})\dfrac{\partial}{\partial
x_{i}}\left(b_{j\,k}(t;{\bf x})\dfrac{\partial u(t;{\bf
x})}{\partial x_{j}}\right)= \\
  =  \dfrac{\partial u(t;{\bf x})}{\partial t}+
  a_{i}(t;{\bf x})\cfrac{\partial u(t;{\bf x})}{\partial
x_{i}}\, +\\
+ \dfrac{1}{2}\,b_{i\,k}(t;{\bf
x})\left[\dfrac{\partial}{\partial x_{i}}\Bigl(b_{j\,k}(t;{\bf x})
\dfrac{\partial u(t;{\bf x})}{\partial x_{j}}\Bigr)-
\dfrac{\partial u(t;{\bf x})}{\partial x_{j}}
\dfrac{\partial b_{jk}(t;{\bf x})}{\partial x_{i}}\right]=\\
= \dfrac{\partial u(t;{\bf x})}{\partial t}+
  a_{i}(t;{\bf x})\cfrac{\partial u(t;{\bf x})}{\partial
x_{i}}\, - \dfrac{1}{2}\,b_{i\,k}(t;{\bf
x})
\dfrac{\partial u(t;{\bf x})}{\partial x_{j}}
\dfrac{\partial b_{jk}(t;{\bf x})}{\partial x_{i}}.
\end{array}
\end{equation*}
Consequently,
$$\displaystyle\frac{\partial u(t;{\bf x})}{\partial
t}+\displaystyle\frac{\partial u(t;{\bf x})}{\partial
x_{i}}\Bigl[a_{i}(t;{\bf x})-
\displaystyle\frac{1}{2}\,b_{j\,k}(t;{\bf x})\frac{\partial
b_{i\,k}(t;{\bf x})}{\partial x_{j}}\Bigr]=0.$$
Thus, we get all the terms of  $\left.\mathcal{L}\right)$.
\end{proof}

\begin{remark}
If we analyzed the concrete realization then the non random function  $u(t;{\bf x})$ is the determinate first integral of the stochastic system.
\end{remark}

\begin{remark}
The concept of the stochastic first integral for  a centered Poisson measure was introduced in {\rm\cite{D_02}}. The obtained conditions for its realization take the necessity of determining the density of intensive Poisson distribution. Thus, it makes no difference what is the probability distribution of intensities of Poisson jumps. This case is very important for constructing program controls {\rm{\cite{11_KchUpr}}}.
\end{remark}

The presented generalization of the It\^{o}-Wentzell formula and the stochastic first integral concept {\rm{\cite{D_02}}}  allow to construct of the program control with the probability is equaled to 1 for dynamic systems which being subjected to the Wiener perturbations and the Poisson jumps {\rm{\cite{11_KchUpr}}}.

\end{document}